%% file: transitivity.tex
\documentclass[11pt,reqno,palentino]{amsart} 

\usepackage{mypack}

\usepackage{latexsym}

\usepackage{a4wide}
\usepackage{amscd}
\usepackage{graphics}
\usepackage{amsmath}
\usepackage{amssymb}
\usepackage{mathrsfs}
\usepackage{accents}
\usepackage{enumerate}
\usepackage{soul,xcolor}
\usepackage{hyperref}
\usepackage{soul}
\usepackage[cal=boondoxo]{mathalfa}

\DeclareMathOperator{\id}{id}

\DeclareMathOperator{\diff}{Diff}
\DeclareMathOperator{\sdim}{sdim}

\newcommand{\cM}{\ensuremath{\mathcal M}}
\newcommand{\cW}{\ensuremath{\mathcal W}}

\newcommand{\cP}{\ensuremath{\mathcal P}}

\newcommand{\cO}{\ensuremath{\mathcal O}}

\author[F. Pasquotto]{Federica Pasquotto}
\author[T.O. Rot]{Thomas O. Rot}

\title{On the transitivity of the group of orbifold diffeomorphisms.}
\begin{document}

\maketitle

\begin{abstract}
Consider a connected manifold of dimension at least two and the group of compactly supported diffeomorphisms that are isotopic to the identity through a compactly supported isotopy. This group acts $n$-transitively: any $n$-tuple of points can be moved to any other $n$-tuple by an element of this group. The group of diffeomorphisms of an orbifold is typically not $n$-transitive: simple obstructions are given by isomorphism classes of isotropy groups of points. In this paper we investigate the transitivity properties of the group of compactly supported diffeomorphisms of orbifolds that are isotopic to the identity through a compactly supported isotopy. We also study an example in the category of area preserving mappings. \end{abstract}

\section{Introduction}

Given a connected manifold $M$ of dimension at least two and a natural number $n$, the group $\diff_{\mathrm c }(M)$ of compactly supported diffeomorphisms of $M$ that are isotopic to the identity through a compactly supported isotopy acts $n$-transitively, meaning that given two $n$-tuples  $(x_1,\ldots,x_n)$ and $(y_1,\ldots,y_n)$ of distinct points in $M$ there exists a diffeomorphism $f\in \diff_{\mathrm c }(M)$ such that $f(x_i)=y_i$ for all $i=1,\ldots,n$. This result is fundamental for many applications in differential topology and holds also when the diffeomorphisms are required to preserve certain structures, e.g. symplectic forms and analytic structures. We refer to~\cite{Michor:1994tg,boothby,banyaga97} for further discussion on this topic.

The orbifold diffeomorphism group was first investigated in~\cite{bb02,bb03}, and it is natural to wonder whether the orbifold diffeomorphism group is $n$-transitive. 
It was observed in \cite[Lemma 5.1]{bb03} that an orbifold diffeomorphism necessarily preserves (up to isomorphism) the isotropy groups of points, hence the group of compactly supported orbifold diffeomorphisms does not act transitively, let alone $n$-transitively. Nevertheless we show that $\diff_{\mathrm c }(\mathcal{O})$, the group of compactly supported orbifold diffeomorphisms that are isotopic to the identity through a compactly supported isotopy, does act transitively on the connected components of each singular stratum.

All the notions which appear in the theorem below will be properly introduced in the next section.

\begin{theorem}[$n$-Transitivity]
  Let $\cO$ be an orbifold. Let $(x_1,\ldots,x_n)$ and $(y_1,\ldots ,y_n)$ be two $n$-tuples of pairwise distinct points.  Assume that for every $i$ the points $x_i$ and $y_i$ lie in the same connected component of the singular stratification. Assume furthermore that each connected component of $\Sigma_1$ contains at most one of the $x_i$'s. Then there exists an orbifold diffeomorphism $f:\cO\rightarrow \cO$, isotopic to the identity through a compactly supported isotopy, such that $f(x_i)=y_i$ for all $1\leq i\leq n$.
\end{theorem}

This has also ramifications in other categories of automorphisms: in Section \ref{sec:applications} we consider an example in the category of orbifolds equipped with an area form. We also discuss a setting where $n$-transitivity might give rise to new invariants of orbifold mappings.

\subsection*{Acknowledgements}

T.O. Rot is supported by NWO-NWA startimpuls - 400.17.608. We thank the referees for carefully reading our manuscript and for their helpful comments.

\section{Orbifolds and orbifold maps}
\label{sec:definitions}
\subsection{Orbifold charts}
Let $O$ be a topological space. 
A (smooth) \emph{orbifold chart} is a triple $(\tilde{U},\Gamma_U,\phi_U)$ which consists of
\begin{itemize}
\item an open, connected subset $\tilde{U}\subseteq \mathbb{R}^n$;
\item a finite group $\Gamma_U$ acting smoothly and effectively on $\tilde{U}$;
\item a continuous map $\phi_U$ from $\tilde{U}$ onto an open subset $U\subseteq O$, such that $\phi_U\circ \gamma=\phi_U$ for all $\gamma\in \Gamma_U$, and $\phi_U$ induces a homeomorphism between $\tilde{U}/\Gamma$ and $U$.
\end{itemize}

An \emph{embedding} between two orbifold charts $(\tilde{U},\Gamma_U,\phi_U)$ and $(\tilde{V},\Gamma_V,\phi_V)$ is a smooth embedding $\lambda: \tilde{U}\rightarrow\tilde{V}$ such that $\phi_V\circ \lambda=\phi_U$. In particular, this means $U=\phi_U(\tilde{U})\subseteq V$.

Two orbifold charts $(\tilde{U},\Gamma_U,\phi_U)$ and $(\tilde{V},\Gamma_V,\phi_V)$ such that $U\cap V\neq \emptyset$ are called \emph{locally compatible} if for every $p\in U\cap V$ there exists an orbifold chart $(\tilde{W},\Gamma_W,\phi_W)$ such that $p\in W\subseteq U\cap V$ and $(\tilde{W},\Gamma_W,\phi_W)$ embeds into both $(\tilde{U},\Gamma_U,\phi_U)$ and $(\tilde{V},\Gamma_V,\phi_V)$.

An \emph{orbifold atlas} on $O$ is a collection $\mathcal{U}$ of orbifold charts such that $\mathcal{U}$ covers $O$, and the elements of $\mathcal{U}$ are locally compatible (in the sense of the above definition). An orbifold atlas $\mathcal{U}$ \emph{refines} and orbifold atlas $\mathcal{V}$ if every chart of $\mathcal{U}$ admits an embedding into some chart of $\mathcal{V}$. Two orbifold atlases are \emph{equivalent} if they admit a common refinement. An effective \emph{orbifold} is a paracompact and Hausdorff space $O$ equipped with an orbifold atlas $\mathcal{U}$. Just as in the manifold case, every orbifold atlas lies in a unique maximal atlas and we will always assume our orbifolds to be equipped with a maximal atlas.


We denote the pair $(O,\mathcal{U})$ by $\mathcal{O}$, to distinguish between the orbifold and its underlying topological space $O$.





The Bochner-Cartan linearization theorem shows that we may always choose charts in which the groups $\Gamma_U$ act linearly on $\tilde U=\mR^n$, i.e.~the chart is a representation of $\Gamma_U$. We call such a chart a linear orbifold chart.

\subsection{Singular set and singular dimension}
Let $\mathcal{O}$ be an orbifold and let $x$ be a point in the underlying topological space $O$. 
If $(\tilde{U},\Gamma_U,\phi_U)$ is any orbifold chart such that $x=\phi_U(\tilde{x})\in U$, we define the isotropy group at $x$ to be
\[
\Gamma_x=\{\gamma\in \Gamma_U\,:\, \gamma\cdot \tilde{x}=\tilde{x}\}.
\]
Up to group isomorphism, this definition is independent of the choice of chart and lift. 
 The \emph{singular set} $\Sigma$ of $\mathcal{O}$ consists of the points with non-trivial isotropy group:
\[
\Sigma=\{x\in O\,:\, \Gamma_x\neq 1\}.
\]
The singular set admits a stratification by singular dimension, which we explain now. Each point $x\in O$ is contained in an open neighborhood $U_x$, for which there exists an orbifold chart 
where the group acting is isomorphic to $\Gamma_x$: we call this a \emph{chart centered at $x$} and write $(\tilde{U}_x,\Gamma_x,\phi_x)$ for such a chart. If $(\tilde{U}_x,\Gamma_x,\phi_x)$ is centered at $x$, then there exists a unique $\tilde x\in \tilde U_x$ such that $\phi_x(\tilde x)=x$. The action of $\Gamma_x$ fixes $\tilde x$, and the differential induces an action on $T_{\tilde x} \tilde U_x$. This action fixes a subspace $T_{\tilde x}\tilde U_x^{\Gamma_x}$, cf.~\cite{D2015}. The \emph{singular dimension} of $x$ is defined to be $\sdim(x)=\dim T_{\tilde x}\tilde U_x^{\Gamma_x}$ and does not depend on the choice of orbifold chart. The singular set can be written as the union $\Sigma=\bigcup_{k= 0}^{n-1} \Sigma_k$ of \emph{singular strata}, where $\Sigma_k=\{x\in \Sigma\,\vert\,\sdim(x)=k\}$.

  The following proposition, which is for example proven in \cite[Proposition 3.4]{DThesis}, uses the fact that that the fixed point set of a smooth finite group action on a manifold is a manifold.
  \begin{proposition}
    Let $\cO$ be an orbifold. Then each singular stratum $\Sigma_k$ is naturally a manifold of dimension $k$, whose tangent space at $x\in \Sigma_k$ is modeled on $T_{\tilde x}\tilde U_x^{\Gamma_x}$, where $\tilde U_x$ is a chart centered at $x$ and $\tilde x$ is the lift of $x$ to this chart. 
  \end{proposition}

We denote by $\Sigma(x)$ the connected component of $\Sigma_{\sdim(x)}$ containing $x$. The isomorphism class of the isotropy group $\Gamma_y$  is constant for $y\in \Sigma(x)$, and if $\cO$ is compact there are finitely many connected components in each singular stratum $\Sigma_k$.






Sometimes we need to consider multiple orbifolds at the same time. Whenever confusion might arise, we decorate the symbols for the singular set etc.~with a sub/superscript specifying the orbifold in question. 

\subsection{Orbifold maps}
Given the two orbifolds $\mathcal{O}$ and $\mathcal{P}$, a \emph{smooth orbifold map} between $\mathcal{O}$ and $\mathcal{P}$ is defined by the following data/conditions:
\begin{itemize}
\item a continuous map $f:O\rightarrow P$ between the underlying topological spaces; 
\item for every $x\in O$, there exist charts $(\tilde{U},\Gamma_U,\phi_U)$ and $(\tilde{V},\Gamma_V,\phi_V)$ containing $x$ and $f(x)$, respectively, with the property that $f$ maps $U$ into $V$ and can be lifted to a smooth map $\tilde{f}:\tilde{U}\rightarrow \tilde{V}$ such that $\phi_V\circ \tilde f=f\circ \phi_U$.
\end{itemize}
\begin{remark}
  There are several, non-equivalent definitions of (smooth) orbifold maps in the literature, such as, for instance, good maps~\cite{CROrbifoldGW}, complete orbifold maps~\cite{BB2012}, and strong maps~\cite{moerpronk}. Conveniently for us, these differences tend to disappear for orbifold diffeomorphisms. For instance, every orbifold diffeomorphism as defined below is a regular map, in the sense of \cite{CROrbifoldGW},  so in particular it is a good orbifold map, with a unique isomorphism class of compatible systems. 
\end{remark}

In this work we are dealing with a special type of orbifold maps, namely orbifold diffeomorphisms, which behave in a particularly nice way: for instance, the lifts to local charts are uniquely determined up to multiplication by the local group. We now recall the definition and some basic properties of orbifold diffeomorphisms.
  
An \emph{orbifold diffeomorphism} is a smooth orbifold map $f:\mathcal{O}\rightarrow \mathcal{P}$ such that there exists another smooth orbifold map $g:\mathcal{P}\rightarrow \mathcal{O}$ with $g\circ f = 1_O$ and $f\circ g=1_P$. An orbifold diffeomorphism is a homeomorphism between the underlying topological spaces and the local lifts of an orbifold diffeomorphism are local diffeomorphisms in the usual sense, as we prove below.

\begin{lemma} 
  If $f:\cO\rightarrow \cP$ is an orbifold diffeomorphism and $\tilde f$ is a local smooth lifting of $f$, 
then $\tilde f$ is a local diffeomorphism. 
\end{lemma}
\begin{proof}
  Let $(\tilde{U},\Gamma_U,\phi_U)$ be an orbifold chart. Let $\tilde \id:\tilde U\rightarrow \tilde U$ be a smooth lift of the identity $\id:U\rightarrow U$. We claim that $\tilde \id$ is a local diffeomorphism: We can assume that the group $\Gamma_U$ acts linearly on $\tilde U=\mR^n$. Let $x\in U\setminus \Sigma$ be a smooth point and let $\tilde x$ be a lift to $\tilde U$ of $x$. Then there exists a unique $\gamma\in \Gamma_U$ such that $\tilde \id(\tilde x)=\gamma \tilde x$. By continuity there exists an open neighborhood $\tilde V$ of $\tilde x$ such that $\tilde \id(\tilde y)=\gamma \tilde y$ for all $\tilde y \in \tilde V$. Thus $T_{\tilde y}\tilde \id=\gamma$ for $\tilde y\in \tilde V$ and $\det(T_{\tilde y}\tilde \id)=\det \gamma$ is constant and non-vanishing on $\tilde V$. The non-smooth points are open and dense, so it follows from the previous argument that  $\det(T_{\tilde y}\tilde \id)$ is constant and non-vanishing on each connected component of the regular set. By continuity, $\det(T_{\tilde y}\tilde \id)$ is constant and non-vanishing on the whole of $\tilde U$. Thus $\tilde \id$ has an invertible differential and is a local diffeomorphism.
    
    Now suppose $f$ is an orbifold diffeomorphism  and let $g$ be the inverse of $f$. Let $\tilde f$ be a local lift of $f$ around $x\in U$ to orbifold charts $(\tilde{U},\Gamma_U,\phi_U)$ and $(\tilde{V},\Gamma_V,\phi_V)$. Let $\tilde g$ be a local lift of $g$ around $y=f(x)$ to orbifold charts $(\tilde{V}',\Gamma_{V'},\phi_{V'})$ and $(\tilde{U'},\Gamma_{U'},\phi_{U'})$. Note that we may, by shrinking the charts if necessary, assume that there exist chart embeddings $\lambda:\tilde V\rightarrow \tilde V'$ and $\mu:\tilde U'\rightarrow \tilde U$. Then the composition $\mu\circ\tilde g\circ\lambda\circ \tilde f:\tilde U\rightarrow \tilde U$ is a lift of the identity, hence by the first part of the proof it is a local diffeomorphism. The underlying spaces $O$ and $P$ are homeomorphic, thus $\dim(\cO)=\dim(\cP)$. If the composition of smooth maps between manifolds of the same dimension are local diffeomorphisms, then each individual map is a local diffeomorphism. Therefore $\tilde g$ and $\tilde f$ are local diffeomorphisms. In fact, by \cite[Proposition 1.11]{DThesis} it then follows that $\tilde \id(x)=\gamma x$ everywhere. 
\end{proof}

The previous lemma implies in particular that by a suitable choice of orbifold charts we may always assume the diffeomorphism $f$ to lift to a diffeomorphism between the charts. That will be our standing assumption in what follows. 
  
\begin{lemma}
  \label{lem:equiv}
If $f$ is an orbifold diffeomorphism, and $\tilde{f}$ is a smooth lifting of $f$ between the orbifold charts $(\tilde{U},\Gamma_U,\phi_U)$ and $(\tilde{V},\Gamma_V,\phi_V)$, then there is an induced group isomorphism $\Theta: \Gamma_U\rightarrow \Gamma_V$ such that $\tilde{f}(\gamma\cdot \tilde{x})=\Theta(\gamma)\cdot\tilde{f}(\tilde{x})$ for all $\tilde{x}\in \tilde{U}$ (in other words, $\tilde{f}$ is $\Theta$-equivariant). If $\tilde f'$ is another lift of $f$ between the same orbifold charts, then $\tilde f=\mu\circ\tilde f'$ for a unique $\mu\in \Gamma_V$. 
\end{lemma}
\begin{proof}
We can construct a map between the groups associated to the orbifold charts in the following way: given $\gamma\in \Gamma_U$, the composition $\tilde{f}\circ \gamma\circ \tilde{f}^{-1}$ is a diffeomorphism of $\tilde{V}$. Moreover, we have that 
\[
\phi_V\circ(\tilde{f}\circ \gamma\circ \tilde{f}^{-1})=f\circ \phi_U\circ \gamma\circ \tilde{f}^{-1}=f\circ \phi_U\circ \tilde{f}^{-1}=\phi_V\circ \tilde{f}\circ \tilde{f}^{-1}=\phi_V,
\] 
so by \cite[Proposition 1.11]{DThesis} it follows that there exists a unique element $\mu\in \Gamma_V$ such that $\tilde{f}\circ \gamma\circ \tilde{f}^{-1}=\mu$. We set $\Theta(\gamma)=\mu$, and this defines a map $\Theta: \Gamma_U\rightarrow \Gamma_V$, which by construction is a group isomorphism.

We note that $\tilde f\circ (\tilde f')^{-1}$ is a lift of the identity on $V$, and again by \cite[Proposition 1.11]{DThesis}, we have that there exists a unique $\mu\in \Gamma_V$ such that $\tilde f\circ (\tilde f')^{-1}=\mu$, i.e.~$\tilde f=\mu\circ \tilde f'$.
\end{proof}


The \emph{support} of a smooth orbifold diffeomorphism $f:\cO\rightarrow \cO$ is the set
  \[
\mathrm{supp}(f)=\overline{\{x\in O\,|\, f(x)\neq x\}}.
  \]

In order to define homotopies of orbifold maps, we need to consider the product orbifold structure on $\mathcal{O}\times [0,1]$ (\cite{BB2008}): its singular set consists of points of the form $(x,t)$, with $x\in \Sigma$ and $t\in [0,1]$, and for all such points the isotropy group $\Gamma_{(x,t)}$ is isomorphic to $\Gamma_x$. The time $t$ inclusion $\mathcal{i}_t:\cO\rightarrow \cO\times [0,1]$ is a smooth orbifold map.

Two orbifold maps $f,g: \mathcal{O}\rightarrow \mathcal{P}$ are called \emph{smoothly homotopic} if there exists a smooth orbifold map $F:\, \mathcal{O}\times [0,1]\rightarrow \mathcal{P}$ such that 
\[
F\mathcal{i}_0=f\quad \textrm{and}\quad F\mathcal{i}_1=g. 
\]
Two orbifold diffeomorphisms are \emph{isotopic} if they are homotopic through orbifold diffeomorphisms. We denote by $\diff_{\mathrm c }(\cO)$ the group of compactly supported diffeomorphisms of the orbifold $\cO$ that are isotopic to the identity through a compactly supported isotopy.


\section{Transitivity}

In this section we will prove our main result. We start with the observation that an orbifold diffeomorphism preserves the singular dimension.

\begin{lemma}\label{Lem:obstruction}
Suppose $f:\cO\rightarrow \cP$ is an orbifold diffeomorphism: then for every $x\in O$ we have  that $\sdim(x)=\sdim(f(x))$.
\end{lemma} 
\begin{proof}
  The isotropy group $\Gamma_x$ acts on $T_{\tilde{x}}\tilde U_x$ via the differential. We denote this action by $\gamma\cdot \tilde X$, where $\gamma\in \Gamma_x$ and $\tilde X\in T_{\tilde{x}}\tilde U_x$. If $\tilde X\in T_{\tilde x}\tilde U_x^{\Gamma_x}$, then $\gamma\cdot \tilde X=\tilde X$ and by the equivariance proved in Lemma~\ref{lem:equiv} we find that $\Theta_x(\gamma)\cdot T_{\tilde x}\tilde f_x(\tilde X)=T_{\tilde x}\tilde f_x(\gamma\cdot \tilde X)=T_{\tilde x}\tilde f_x(x) \tilde X$. Since the map $\Theta_x:\Gamma_x\rightarrow \Gamma_{f(x)}$ is a group isomorphism, every $\mu\in \Gamma_{f(x)}$ is of the form $\Theta_x(\gamma)$ for some $\gamma\in \Gamma_x$. Thus $\mu\cdot T_{\tilde x}\tilde f_x\tilde X=T_{\tilde x}\tilde f_x(\tilde X)$ for all $\mu\in \Gamma_{f(x)}$.  We conclude that $T_{\tilde x}\tilde f_x (T_{\tilde x} \tilde U_x^{\Gamma_x})\subset T_{\tilde f_x(\tilde x)}\tilde U_{f(x)}^{\Gamma_{f(x)}}$, which implies that $\sdim(x)\leq \sdim(f(x))$. The same argument applied to the inverse map $f^{-1}$ then shows that $\sdim(f(x))\leq \sdim(x)$, hence $\sdim(x)=\sdim(f(x))$. 
\end{proof}

The next proposition is the crucial step in the proof of our result: it shows that locally we can carry a given point $x$ to any sufficiently close point lying on the same connected component by a diffeomorphism which is isotopic to the identity via an isotopy having compact support in a prescribed open neighborhood of $x$.

\begin{proposition}
  \label{prop:open}
  Let $\cO$ be an orbifold, and let $U$ be an open neighborhood of a point $x\in O$. Then there exists an open subset $V$ of $\Sigma(x)$, such that $x\in V$, $V\subseteq U$, and for each $y\in V$ there exists $f\in \diff_{\mathrm c}(\cO)$ such that $f(x)=y$ and $f$ is isotopic to the identity via an isotopy with compact support contained in $U$. 
\end{proposition}

\begin{proof}
  Choose an orbifold chart $(\tilde U_x=\mathbb{R}^n,\Gamma_x,\phi_x)$ centered around $x$, and such that $U_x\subseteq U$. Let $\tilde{x}\in\tilde{U}_x$ be the unique point such that $\phi_x(\tilde{x})=x$, and let $\tilde V\subseteq \tilde U_x$ be the connected component of the fixed set of $\Gamma_x$ containing $\tilde x$. Define $V=\phi_x(\tilde V)$: the restriction of the map $\phi_x$ to $\tilde V$ is injective, and $\phi_x(\tilde V)\subseteq\Sigma(x)\cap U_x$. Let $\tilde y\in \tilde V$. We want to show that that exists a diffeomorphism $\tilde f$ of $\tilde U_x$ which satisfies the following three conditions: 
  \begin{itemize}
\item[(i)]  $\tilde f$ is equivariant with respect to the action of $\Gamma_x$;
\item[(ii)] $\tilde f$ is isotopic to the identity, via equivariant diffeomorphisms compactly supported in $U$;
\item[(iii)] $\tilde f(\tilde x)=\tilde y$.
\end{itemize}

Since $\tilde V$ is a connected manifold, there exists a vector field $\tilde Y$ on $\tilde V$, which is compactly supported\footnote{A vector field is compactly supported if it vanishes outside a compact set} and whose time-one flow maps $\tilde x$ to $\tilde y$. We can extend $\tilde Y$ to a compactly supported vector field defined everywhere on $\tilde U_x$, and we will denote this vector field by the same symbol. Recall that the action of $\Gamma_x$ on $\tilde U_x$ induces a linearized action on $T\tilde U_x$, which we denote by $\gamma\cdot \tilde X$, with $\gamma\in \Gamma_x$ and $\tilde X\in T\tilde U_x$. Note that $\gamma\cdot:T_{\tilde z}\tilde U_x\rightarrow T_{\gamma\tilde z}\tilde U_x$ and, as $\tilde Y$ is tangent to $\tilde V$ and the action fixes $\tilde V$, that $\gamma\cdot \tilde Y(\tilde z)=\tilde Y(\tilde z)$ for all $\tilde z\in \tilde V$. We need to average over the action of the group $\Gamma_x$ in order to make our vector field equivariant, that is, for all $\tilde z \in \tilde U_x$ we define $\tilde X(\tilde z)$ via the formula:
\[
\tilde X(\tilde z)=\frac{1}{\vert \Gamma_x\vert}\sum_{\gamma\in \Gamma_x}\gamma\cdot \tilde Y(\gamma^{-1} \tilde z).
\]
 Since $\tilde Y$ is compactly supported, and $\Gamma_x$ is finite, the vector field $\tilde X$ is compactly supported. Then we check that, for $\mu\in \Gamma_x$, 
\begin{align*}
  \mu\cdot \tilde X(\tilde z)&=\frac{1}{\vert \Gamma_x\vert}\sum_{\gamma\in \Gamma_x} \mu\cdot (\gamma \cdot \tilde Y(\gamma^{-1} \tilde z))\\
  &=\frac{1}{\vert \Gamma_x\vert}\sum_{\gamma\in \Gamma_x} (\mu \gamma)\cdot \tilde Y((\mu \gamma)^{-1}\mu \tilde z)=\frac{1}{\vert \Gamma_x\vert}\sum_{\gamma '\in \Gamma_x} \gamma '\cdot \tilde Y(\gamma'^{-1} \mu \tilde z)=\tilde X(\mu \tilde z),
\end{align*}
i.e. $\tilde X$ is indeed equivariant with respect to the action of $\Gamma_x$. Note that for $\tilde z\in \tilde V$ we have that $\gamma\cdot X(\tilde z)=\tilde Y(\tilde z)$, hence the flow of $\tilde X$ agrees with that of $\tilde Y$ when restricted to $\tilde V$. In particular, the time-one flow $\tilde f$ of $\tilde X$ maps $\tilde x$ to $\tilde y$. Clearly the projection of $\tilde f$ to $U$ can be extended (by declaring it to be the identity outside of $U$) to an orbifold diffeomorphism $f$, which is isotopic to the identity through an isotopy with support in $U$. By construction, $f$ maps $x$ to $y=\phi_x(\tilde y)$. 
\end{proof}

We can now move on to the proof of our main result.

\begin{theorem}[Transitivity]
  \label{thm:ktransitivity}
  Let $\cO$ be a connected orbifold. Let $(x_1,\ldots,x_n)$ and $(y_1,\ldots ,y_n)$ be two $n$-tuples of pairwise distinct points.  Assume that for every $i$ the points $x_i$ and $y_i$ lie in the same connected component of the singular stratification. Assume furthermore that each connected component of $\Sigma_1$ contains at most one of the $x_i$'s. Then there exists an orbifold diffeomorphism $f:\cO\rightarrow \cO$, isotopic to the identity through a compactly supported isotopy, such that $f(x_i)=y_i$ for all $1\leq i\leq n$.
  \end{theorem}

  \begin{proof}
We will start by reordering the $x_i$'s and $y_i$'s according to which stratum they belong to. Up to a permutation, we may assume that there exist numbers $1=l_1<l_2<\ldots<l_k\leq n$ such that for all $i\leq k-1$ one has $x_j\in \Sigma^\cO(x_{l_i})=:Y_i$ if and only if $l_i\leq j\leq l_{i+1}-1$, and $x_j\in \Sigma^\cO(x_{l_k})=:Y_k$ if and only if $j\geq l_{k}$ (and analogously for the $y_i$'s). Consider the $n$-fold cartesian product $\cO^n=\cO\times \ldots \times \cO$ and the (open) suborbifold $\cO^{(n)}\subseteq \cO^n$, whose underlying space consists of $n$-tuples of points which are pairwise distinct, i.e.
$$
O^{(n)}=\{(x_1,\ldots,x_n)\in O^n\,|\, x_i\not=x_j\quad \text{if}\quad i\not=j\}. 
$$
Let $n_i=l_{i+1}-l_{i}$ for all $1\leq i\leq k-1$ and $n_k=n-l_k+1$. Then both $(x_1,\ldots,x_n)$ and $(y_1,\ldots ,y_n)$ lie in the singular stratum $\Sigma^{\cO^n}(x_1,\ldots,x_n)=Y_1^{n_1}\times \ldots\times Y_k^{n_k}$. We claim that the following identity holds:   \begin{equation}
  \label{eq:connected}
  \Sigma^{\cO^{(n)}}(x_1,\ldots, x_n)=\Sigma^{\cO^n}(x_1,\ldots,x_n)\cap \cO^{(n)}.
\end{equation}
The inclusion of the left hand side in the right hand side is immediate, and for the inclusion of the right hand side in the left hand side we should check that $\Sigma^{\cO^n}(x_1,\ldots,x_n)\cap \cO^{(n)}$ is connected. The fat diagonal in $Y_i^{n_i}$ is defined by
\[
  \Delta_i=\{(z_{1},\ldots, z_{n_i})\in Y_i^{n_i}\,| z_j=z_k \text{ for some } j\not =k \}.
\]
By our assumption on $\Sigma_1$, the fat diagonal $\Delta_i$ is either empty or a union of submanifolds of $Y_i^{n_i}$ of codimension greater than or equal to two in $Y_i^{n_i}$. Thus $Y_i^{n_i}\setminus \Delta_i$ is connected as $Y_i^{n_i}$ is connected. By a similar argument \[
  \Sigma^{\cO^n}(x_1,\ldots,x_n)\cap \cO^{(n)}= Y_1^{n_1}\times \ldots \times Y_k^{n_k}\setminus \bigcup_{i=1}^k\left(Y_1^{n_1}\times \ldots \times \Delta_i\times\ldots Y_k^{n_k}\right)\]
is connected and~\eqref{eq:connected} is true.

The group $\diff_{\mathrm c}(\cO)$ acts on $\cO^{(n)}$ diagonally and this action preserves each singular stratum $\Sigma^{\cO^{(n)}}(x_1,\ldots,x_n)$. We show next that each $\diff_{\mathrm c }(\cO)$-orbit is open in $\Sigma^{\cO^{(n)}}(x_1,\ldots,x_n)$. Let $(z_1,\ldots,z_n)\in \Sigma^{\cO^{(n)}}$ and choose for each $z_i$ an open subset $U_i\ni z_i$ such that $U_i\cap U_j=\emptyset$ if $i\not=j$. Then $U_1\times \ldots \times U_n\subseteq O^{(n)}$. Let $V_i$ be the open subset of $\Sigma^{\cO}(z_i)$ given by Proposition~\ref{prop:open}, where we choose $x=z_i$ and $U=U_i$. Then for each $w_i\in V_i$ there exists $f_i\in \diff_{\mathrm c }(\cO)$ such that $f_i(z_i)=w_i$. Let $g=f_1\circ \ldots \circ f_n$. Since all $f_i$'s have distinct support, we have that $g(z_i)=w_i$. It follows that the $\diff_{\mathrm c}(\cO)$-orbit through each point $(z_1,\ldots,z_n)$ is open in $\Sigma^{\cO^{(n)}}(x_1,\ldots ,x_n)$. Since $\Sigma^{\cO^{(n)}}(x_1,\ldots ,x_n)$ is connected, there can be only one orbit and hence there exists an element $f\in \diff_{\mathrm c }(\cO)$ such that $f(x_1,\ldots,x_n)=(y_1,\ldots, y_n)$.
\end{proof}

\begin{remark}
  The statement can be improved to include multiple points in each connected component of $\Sigma_1$, but the result is more technical to state. The reason is that even in the manifold case the diffeomorphism group does not act $k$-transitively on points if the manifold is one dimensional and $k>1$. Consider for example the one-dimensional manifold $\mR$. Then two tuples $(x_1,\ldots,x_n)$ and $(y_1,\ldots,y_n)$ are said to be \emph{ordered} if $x_1<\ldots <x_n$ and $y_1<\ldots<y_n$. If the $n$-tuples $(x_1,\ldots,x_n)$ and $(y_1,\ldots,y_n)$ are both ordered, then there exists a diffeomorphism $f\in \diff_{\mathrm c }(\mR)$ with $f(x_i)=y_i$, whereas if only one of the $n$-tuples is ordered and the other is not, such a diffeomorphism does not exist. If one of the tuples has the reverse order, e.g. $y_1>y_2>\ldots>y_n$, a diffeomorphism still exists, but it is not isotopic to the identity, as it reverses the orientation of $\mR$. 

  The case for the circle $S^1$ is similar: one can lift a diffeomorphism of $S^1$ to a diffeomorphism of $\mR$ by choosing a base-point in $S^1$ and an initial lift to $\mR$. The points $x_1,\ldots,x_n$ and $y_1,\ldots , y_n$ are then lifted to points $\tilde x_1,\ldots ,\tilde x_n$ and $\tilde y_1,\ldots ,\tilde y_n$. There exists a diffeomorphism that maps $x_i$ to $y_i$ if and only if there exists cyclic permutations $\sigma$ and $\mu$ such that
\[
  \tilde x_{\sigma(1)}<\ldots <\tilde x_{\sigma(n)}\quad \text{and}\quad \tilde y_{\mu(1)}<\ldots <\tilde y_{\mu(n)}.
\]

In the orbifold case each connected component of $\Sigma_1$ is a one dimensional manifold, and we can still make sense of the orderability condition above by orienting these manifolds. We do not think it worthwhile to make a full statement.
\end{remark}

\begin{corollary}
Let $\cO$ be an orbifold with $\dim\cO\geq 2$ and let $x\in \cO$. Then the orbit $\diff_{\mathrm c} (\cO)x$ coincides with $\Sigma(x)$. Hence the orbit $\diff_{\mathrm c }(\cO)x$ is dense in a connected orbifold $\cO$ if and only if $\sdim(x)=\dim \cO$, that is, $x$ is a non-singular point. 
\end{corollary}

\begin{remark}
  The referee posed the following interesting question: are non-effective orbifolds also transitive, in the sense discussed in this paper? We do not know the answer at the time of writing.
\end{remark}

\section{Applications}\label{sec:applications}
\subsection{Displacing curves on a two sphere, and its orbifold analogue.}

The non-transitivity of the orbifold diffeomorphism group has consequences outside of the smooth category, even if we are not concerned with the displacement of points: for instance, in the setting of area preserving maps. We give an example here, and we are confident that the reader can think of more applications.

Let us recall the following property of area preserving maps of the two-sphere. Equip the sphere $S^2$ with the standard area form $\omega$. Consider an embedded simple closed curve $i:S^1\rightarrow S^2$, and let $u_1:D^2\rightarrow S^2$ and $u_2:D^2\rightarrow S^2$ be the two discs bounded by this curve\footnote{It is clear that the curve divides $S^2$ into two compact orientable surfaces with boundary. A simple argument involving the Euler characteristic and the classification of orientable surfaces then shows that these must be discs. }. If we denote by $D_1$ and $D_2$ their image on the sphere, that is, $D_i=u_i(D^2)$, $i=1,2$, then their area is given by $A_1=\int_{D^2}u_1^*\omega$ and $A_2=\int_{D^2}u^*_2\omega$, respectively.
We are going to prove that any such simple closed curve corresponds, under an area preserving diffeomorphism of the sphere, to a circle of constant height.
Consider then such a circle of latitude $L$, chosen so that it divides the sphere into two spherical caps $C_1$ and $C_2$, of area $A_1$ and $A_2$, respectively. 
Identifying $S^1$ with $L$, and using Alexander's trick, we find a diffeomorphism $f:S^2\rightarrow S^2$ which carries the circle of latitude $L$ to the image of $i:S^1\rightarrow S^2$, our original simple closed curve. See Figure~\ref{fig:moser}. The area form $\omega^*=f^*\omega$ is not standard, but the areas of $C_1$ and $C_2$, with respect to $\omega$ and $\omega^*$, are the same. Moreover, in view of the Lagrangian tubular neighborhood theorem, we may assume $\omega$ and $\omega^*$ to already coincide in an open neighborhood of $L$, that is, we may assume $\omega-\omega^*$ to have support in the interior of $C_1$ and $C_2$. Now an application of Moser's stability argument for volume forms \cite{moser} shows that there exists a diffeomorphism of the sphere, fixing the circle $L$, which identifies $\omega^*$ and $\omega$.

If $A_1$ and $A_2$ are not equal, the original circle is mapped to a circle latitude strictly above or below the equator. Via a rotation, which is an area preserving diffeomorphism, the curve can be displaced from itself. Note that the disc with smaller area is mapped into the disc with larger area.

If $A_1=A_2$, the embedded curve is called \emph{monotone} (or \emph{balanced}), and it cannot be displaced by an area preserving map. In fact, by the argument above, we may always assume the image of the curve to be the equator. If the equator could be displaced by an area preserving map, then one hemisphere would have to be mapped onto a proper subset of its interior or of the interior of the opposite hemisphere, contradicting in both cases the assumption that the map is area preserving.~\footnote{A nice example of an area preserving map is the antipodal map $-\id$. Thus any simple closed curve dividing $S^2$ into two regions of the same area, must have antipodal points. A nice visual explanation is given in the numberphile video~\cite{numberphile}}.
\begin{figure}

\def\svgwidth{.5\textwidth}
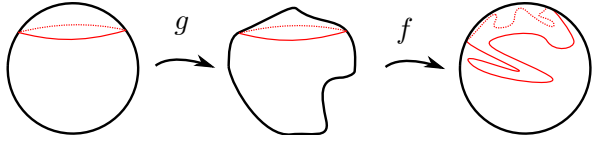
\caption{On the right a simple closed curve on $S^2$ is drawn. Via a diffeomorphism $f$ this is pulled back to a parallel curve. Pulling the standard area form on the right back to the middle picture creates a non-standard area form. Via an additional diffeomorphism the area form can be brought back into standard form. The original curve is pulled back to a circle of constant latitude. The height is completely determined by the ratio of areas $A_1/A_2$. }
\label{fig:moser}
\end{figure}

This discussion can be summarized in the statement below:

\begin{proposition}
A simple closed curve on the two sphere can be displaced from itself by an area preserving diffeomorphism if and only if the curve is not monotone. 
\end{proposition}
There are various ways to generalize this example. For instance, this result fits into the framework of non-displaceable (monotone) Lagrangian submanifolds, which is due to Floer \cite{floerintersections} (in the exact case)  and Oh~\cite{Oh:1993ek} in the monotone case. In this generality the proofs are much more difficult. 

Let us now move on to the orbifold world. Let $p$ and $q$ be two coprime natural numbers. The $(p,q)$-spindle is the orbifold obtained as $\mathbb{CP}^1(p,q)=S^3//S^1$, where $S^1\subseteq \mC$ acts on $S^3\subseteq \mC^2$ via $t\cdot (z_0,z_1)=(t^pz_0,t^q z_1)$. As this action is isometric with respect to the standard Riemannian metric, the $(p,q)$-spindle is naturally a Riemannian orbifold, and therefore also carries an area form. The argument that follows does not depend on the specific area form. We denote the points on $\mathbb{CP}^1(p,q)$ by the equivalence class $[z_0,z_1]$. The spindle has two isolated singular points at $N=[0,1]$ and $S=[1,0]$, with isotropy $\mathbb{Z}_p$ and $\mathbb{Z}_q$, respectively. Now let $\gamma$ be any simple closed curve that does not pass through the singular points $N$ and $S$. The curve divides the spindle into two parts $D_1$ and $D_2$. There are a few cases to consider. Suppose $N\in D_1$ and $S\in D_2$. Then there is no way to displace the curve from itself by an area preserving map: if $f$ is a diffeomorphism, it must fix $N$ and $S$ by Lemma~\ref{Lem:obstruction}. The curve must then be mapped to the interior of either $D_1$ or $D_2$. Without loss of generality we can assume that it is mapped to the interior of $D_1$. Since the singular point $N$ is fixed, $D_1$ should be mapped to its own interior. But this is clearly not possible in an area preserving manner.

\begin{proposition}
Suppose $p$ and $q$ are coprime. If a simple closed curve on the spindle $\mathbb{CP}^1(p,q)$ divides it into two components, each containing one of the singular points, then it cannot be displaced from itself by an area preserving orbifold diffeomorphism. 
\end{proposition}

The interesting thing to note is that we are trying to displace a curve that lies entirely in the manifold part of the orbifold, but it still ``feels'' the singularities far away.

If both singular points lie in one region, say $N,S\in D_1$, then we can displace the curve by an area preserving map, provided the area of $D_2$ is smaller than that of $D_1$. 

\subsection{Homotopy classes of maps}

Distinguishing homotopy classes of maps between connected closed manifolds $M$ and $N$ is an interesting and challenging problem. One of the easiest invariants which can be defined to approach this problem is the (embedded) cobordism class of a regular value. A regular value $y$ of a map $f:M\rightarrow N$ defines a submanifold $f^{-1}(y)$ of $M$. If $F: M\times [0,1]\rightarrow N$ is a homotopy between $f$ and $g$, and $y$ is a regular value $y$ of $F$, which is also a regular value of $f$ and $g$, then $F$ defines an embedded cobordism $F^{-1}(y)$ in $M\times [0,1]$ between $f^{-1}(y)$ and $g^{-1}(y)$. 

It follows from the fact that regular values are open and dense, and that the diffeomorphism group acts transitively on $N$, that the cobordism class of a regular value of a map $f:M\rightarrow N$ is independent of the regular value and the homotopy class of $f$.

But this invariant is not strong enough to distinguish some important maps. Let us consider for instance the Hopf fibration, which is homotopically non-trivial. If we view $S^3$ as the unit sphere in $\mathbb{C}^2$, and $S^2$ as the unit sphere in $\mathbb C\times \mR$, then the Hopf fibration is the map $f:S^3\rightarrow S^2$ defined by $f(z_0,z_1)=(2z_0\overline z_1,\vert z_0\vert^2-\vert z_1\vert ^2)$. This map is in fact a generator of $\pi_3(S^2)\cong \mZ$. Every value is regular and the preimage of each regular value is an unknotted, embedded circle in $S^3$. The unknot is the boundary of an embedded disc in $S^3\times [0,1]$, hence the cobordism class of the preimage of a regular value does not distinguish the homotopy class of the Hopf fibration from the homotopy class of the constant map.

Still we can try to obtain further information by considering a pair of distinct points and their preimages. In this setting, a cobordism will consist of a pair of disjoint, embedded surfaces, whose boundary consist of the preimages of different pairs of regular points. Given that the diffeomorphism group acts $2$-transitively on $S^2$, the embedded cobordism class of two regular values is independent of the chosen regular values, and it is independent of the homotopy class. The preimages of two points $p,q\in S^2$ under the Hopf map are two embedded unknots in $S^3$, but they are \emph{linked}. In particular it can be shown that that we cannot find a pair of disjoint, embedded surfaces in $S^3\times[0,1]$ whose boundary consists of these two circles alone. It follows that the Hopf fibration is homotopically non-trivial. In this case this invariant captures the same information as the \emph{framed cobordism} sets of $S^3$, cf.~\cite{milnor1997}, but in general these invariants are different. 

Let us formalize the argument above directly in the language of orbifolds. Given a proper orbifold map $f:\cO\rightarrow \cP$, and a choice of connected components of some singular strata $\Sigma(x_1),\ldots \Sigma(x_n)$, we would like to associate a proper homotopy invariant to $f$. To make sure that regular values are well-defined, and the preimages of regular values are well-defined orbifolds, we need to impose additional structure on the orbifold mappings, and consider for instance complete orbifold mappings \cite{BB2012} or good maps \cite{CROrbifoldGW}. Moreover, we need to define the notion of full suborbifolds. We will not discuss these notions here, but we refer the reader to~\cite{weilandt,mestreweilandt,bb17} for more details. A $k$-dimensional \emph{$n$-colored suborbifold}, is a collection of $n$ disjoint compact full suborbifolds (without boundary) $(\cN_1,\ldots, \cN_n)$ of $\cO$. A cobordism between the $k$-dimensional $n$-colored suborbifolds $(\cN_1,\ldots, \cN_n)$ and $(\cM_1,\ldots, \cM_n)$ is a tuple $(\cW_1,\ldots,\cW_n)$ of $(k+1)$-dimensional compact full suborbifolds of $\cO\times [0,1]$ with boundary $(\cN_1\times\{0\},\ldots, \cN_n\times\{0\})\cup (\cM_1\times\{1\},\ldots, \cM_n\times\{1\})$. We denote the set of $k$-dimensional $n$-colored suborbifolds modulo cobordism by $\mathrm{CEmb}^n_k(\cO). $
Note that the suborbifolds $\cN_i$, $\cM_i$ and $\cW_i$ are not required to be connected. Now perturb $f$  to $g$ such that there exist regular values $y_i\in \Sigma(x_i)$, and define the $n$-colored suborbifold $(\cN_1=g^{-1}(y_1),\ldots ,\cN_n= g^{-1}(y_n))$. The invariance, up to $n$-colored cobordism,  under different choices of regular values now depends on the $n$-transitivity of the compactly supported orbifold diffeomorphism group. But in this orbifold setting, we have to deal with the problem of equivariant transversality: even though the set of regular values of an orbifold mapping is still dense~\cite{BB2012}, if a point $y$ lies in the singular stratum of $P$ it is not clear that the map $f:\cO\rightarrow \cP$ can be made transverse to $y$ by a suitable perturbation. We managed to avoid these transversality issues when $\dim(\cO)=\dim(\cP)$ and developed a degree theory for orbifolds~\cite{pasquottorot}. The general case remains open. 

\noindent

\bibliographystyle{abbrv} 
 \bibliography{transitivity}

\end{document}

%% file: moser.pdf_tex
\begingroup%
  \makeatletter%
  \providecommand\color[2][]{%
    \errmessage{(Inkscape) Color is used for the text in Inkscape, but the package 'color.sty' is not loaded}%
    \renewcommand\color[2][]{}%
  }%
  \providecommand\transparent[1]{%
    \errmessage{(Inkscape) Transparency is used (non-zero) for the text in Inkscape, but the package 'transparent.sty' is not loaded}%
    \renewcommand\transparent[1]{}%
  }%
  \providecommand\rotatebox[2]{#2}%
  \ifx\svgwidth\undefined%
    \setlength{\unitlength}{772.07095649bp}%
    \ifx\svgscale\undefined%
      \relax%
    \else%
      \setlength{\unitlength}{\unitlength * \real{\svgscale}}%
    \fi%
  \else%
    \setlength{\unitlength}{\svgwidth}%
  \fi%
  \global\let\svgwidth\undefined%
  \global\let\svgscale\undefined%
  \makeatother%
  \begin{picture}(1,0.22709351)%
    \put(0,0){\includegraphics[width=\unitlength,page=1]{moser.pdf}}%
    \put(0.11262437,0.06043006){\color[rgb]{0,0,0}\makebox(0,0)[lb]{\smash{}}}%
    \put(0.28729374,0.18107036){\color[rgb]{0,0,0}\makebox(0,0)[lb]{\smash{$g$}}}%
    \put(0.66623742,0.16626785){\color[rgb]{0,0,0}\makebox(0,0)[lb]{\smash{$f$}}}%
    \put(0,0){\includegraphics[width=\unitlength,page=2]{moser.pdf}}%
  \end{picture}%
\endgroup%